\documentclass[11pt,bezier]{article}
\usepackage{amsmath,amssymb,graphicx}

\newtheorem{prethm}{{\bf Theorem}}

\newenvironment{thm}{\begin{prethm}{\hspace{-0.5
               em}{\bf.}}}{\end{prethm}}

\newtheorem{prepro}[prethm]{Proposition}

\newtheorem{prelem}[prethm]{Lemma}
\newenvironment{lem}{\begin{prelem}{\hspace{-0.5
               em}{\bf.}}}{\end{prelem}}

\newtheorem{precor}[prethm]{Corollary}

\newtheorem{prerem}[prethm]{{\bf Remark}}

\newtheorem{preexample}{{\bf Example}}

\newtheorem{preproof}{{\bf Proof.}}

\newenvironment{proof}[1]{\begin{preproof}{\rm
               #1}\hfill{$\Box$}}{\end{preproof}}

\newcommand{\noi}{\noindent}

\newcommand{\te}{{\theta}}
\newcommand{\x}{{\bf x}}
\newcommand{\w}{{\bf w}}
\newcommand{\g}{\geqslant}
\newcommand{\lee}{\leqslant}

\newcommand{\e}{{\bf e}}

\newcommand{\s}{{\cal S}}

\newcommand{\la}{\lambda}
\newcommand{\al}{\alpha}
\newcommand{\be}{\beta}

\renewcommand{\thefootnote}

\title{On eigenvalues of Seidel matrices and Haemers' conjecture}
\author{ Ebrahim Ghorbani\\ [.3cm]
{\small\sl Department of Mathematics, K.N. Toosi University of Technology,}\\
{\small\sl P. O. Box 16315-1618, Tehran, Iran}\\
{\small\sl School of Mathematics, Institute for Research in Fundamental
Sciences (IPM),}\\{\small\sl P.O. Box
19395-5746, Tehran, Iran }
\\[.3cm]{
 $\mathsf{e\_ghorbani@ipm.ir}$ }}

\date{}

\begin{document}
\maketitle

\begin{abstract}
For a graph $G$, let $S(G)$ be the Seidel matrix of $G$ and $\te_1(G),\ldots,\te_n(G)$ be the eigenvalues of $S(G)$.
The Seidel energy of $G$ is defined as $|\te_1(G)|+\cdots+|\te_n(G)|$.
Willem Haemers conjectured that the Seidel energy of any graph with $n$ vertices is at least $2n-2$, the
 Seidel energy of the complete graph with $n$ vertices.
Motivated by this conjecture, we prove that for any $\al$ with $0<\al<2$, $|\te_1(G)|^\al+\cdots+|\te_n(G)|^\al\g (n-1)^\al+n-1$ if and only if $|{\rm det}\,S(G)|\g n-1$.  This, in particular, implies the Haemers' conjecture for all graphs $G$ with $|{\rm det}\,S(G)|\g n-1$.

\vspace{3mm}
\end{abstract}

\section{Introduction}

Let $G$ be a simple graph with vertex set $\{v_1,\ldots,v_n\}$.  The {\em Seidel matrix} of $G$ is an $n\times n$ matrix $S(G)=(s_{ij})$
where $s_{11}=\cdots=s_{nn}=0$ and for $i\ne j$,
$s_{ij}$ is $-1$ if $v_i$ and $v_j$ are adjacent, and is
$1$ otherwise. The {\em Seidel energy}  of $G$, denoted by $\s(G)$, is defined as the sum of the absolute values of the eigenvalues of $S(G)$.

Considering the complete graph $K_n$, its Seidel matrix is $I-J$. Hence the eigenvalues of $S(K_n)$ are $1-n$ and $1$ (the latter with multiplicity $n-1$).
So $\s(K_n)=2n-2$. Haemers conjectured that this is the smallest Seidel energy of an $n$-vertex graph:

\noindent{\bf Conjecture} (Haemers \cite{ham}){\bf.} {\em For any graph $G$ on $n$ vertices, $\s(G)\g\s(K_n)$.}

We show that the conjecture is true if $|{\rm det}\,S(G)|\g|{\rm det}\,S(K_n)|=n-1$. To be more precise, we prove the following more general statement which makes
the main result of the present paper.

\begin{thm}\label{main}  Let $G$ be a graph with $n$ vertices and let $\te_1,\ldots,\te_n$ be the eigenvalues of $S(G)$.
Then the following are equivalent:
\begin{itemize}
  \item[\rm(i)] $|{\rm det}\,S(G)|\g n-1;$
  \item[\rm(ii)] for any $0<\al<2$,
\begin{equation}\label{maineq}
    |\te_1|^\al+\cdots+|\te_n|^\al\g(n-1)^\al+(n-1).
\end{equation}
\end{itemize}
\end{thm}
The implication `(ii)$\Rightarrow$(i)' is strightforward in view of the fact that
$$\lim_{\al\to0^+}\left(\frac{|\te_1|^\al+\cdots+|\te_n|^\al}{n}\right)^\frac{1}{\al}=|\te_1\cdots\te_n|^\frac{1}{n}.$$
We prove the implication `(i)$\Rightarrow$(ii)' in Section~3. The proof is based on KKT method in nonlinear programming.
 We briefly explain this method in Section~2.



For more results of the same flavor as (\ref{maineq}) on Laplacian and signless Laplacian eigenvalues of graphs see \cite{agko1,agko2}.

\section{Karush--Kuhn--Tucker (KKT) conditions}

In nonlinear programming, the Karush--Kuhn--Tucker (KKT) conditions are
necessary for a local solution to a minimization problem provided that some regularity conditions are
satisfied. Allowing inequality constraints, the KKT approach to nonlinear programming generalizes the method of
Lagrange multipliers, which allows only equality constraints. For details see \cite{book}.

Consider the following optimization problem:
\begin{quote}
   Minimize $f(\x)$\\
   subject to:\\
   $~~~~~g_j(\x)=0$,~ for $j\in J$,\\
   $~~~~~h_i(\x)\lee0$,~ for $i\in I$,
\end{quote}
where $I$ and $J$ are finite sets of indices.
Suppose that the objective function $f:\mathbb{R}^n\to\mathbb{R}$ and the constraint functions $g_j:\mathbb{R}^n\to\mathbb{R}$ and
$h_i:\mathbb{R}^n\to\mathbb{R}$ are continuously differentiable at a point $\x^*$. If $\x^*$ is a local minimum that satisfies some regularity conditions, then there exist constants $\mu_i$ and $\la_j$, called
KKT multipliers, such that
\begin{align*}
\nabla f(\x^*)+\sum_{j\in J}&\,\mu_j\nabla g_j(\x^*)+\sum_{i\in I}\la_i\nabla h_i(\x^*)={\bf0}\\
  g_j(\x^*)&=0,~~~\hbox{for all $j\in J$},\\
   h_i(\x^*)&\lee0,~~~\hbox{for all $i\in I$},\\
    \la_i&\g0,~~~\hbox{for all $i\in I$},\\
     \la_ih_i(\x^*)&=0,~~~\hbox{for all $i\in I$}.
\end{align*}

In order for a minimum point to satisfy the above KKT conditions, it should satisfy some regularity conditions (or constraint qualifications).  The one which suits our problem is the Mangasarian--Fromovitz constraint qualification (MFCQ).
Let $I(\x^*)$ be the set of indices of active inequality constraints at $\x^*$, i.e.
$I(\x^*)=\left\{i\in I\mid  h_i(\x^*)=0\right\}$. We say that MFCQ holds at a feasible point $\x^*$
if the set of gradient vectors
$\{\nabla g_j(\x^*)\mid j\in J\}$ is linearly independent and that there exists $\w\in\mathbb{R}^n$ such that
\begin{align*}
 \nabla g_j(\x^*)\w^\top&=0,~~~ \hbox{for all $j\in J$},\\
  \nabla h_i(\x^*)\w^\top&<0,~~~  \hbox{for all $i\in I(\x^*)$}.
  \end{align*}

\begin{thm} {\rm(\cite{mf}, see also \cite[Section 12.6]{book})} If a local minimum $\x^*$ of the function  $f(\x)$ subject to the constraints $g_j(\x)=0$, for $j\in J$, and   $h_i(\x)=0$, for $i\in I$, satisfies MFCQ, then it satisfies the KKT conditions.
\end{thm}

\section{Proofs}


In this section we prove the non-trivial part of Theorem~\ref{main}, that is the implication `(i)$\Rightarrow$(ii)'.
We formulate this as an optimization problem. To this end, we need to come up with appropriate constraints.
The main constraint is made by the assumption $|{\rm det}\,S(G)|\g n-1$. The other ones are obtained  by the following straightforward lemma.

\begin{lem} For any graph $G$ with $n$ vertices, we have
\begin{itemize}
  \item[\rm(i)]  $\te_1(G)^2+\cdots+\te_n(G)^2=(n-1)^2+n-1;$
  \item[\rm(ii)] $\te_1(G)^4+\cdots+\te_n(G)^4\lee \te_1(K_n)^4+\cdots+\te_n(K_n)^4=(n-1)^4+n-1;$
  \item[\rm(iii)]  ${\displaystyle\max_{1\lee i\lee n}}\te_i(G)^2\lee{\displaystyle\max_{1\lee i\lee n}}\te_i(K_n)^2=(n-1)^2$.
\end{itemize}
\end{lem}

Now, we can describe our problem as the minimization
of the function  $$f(\x):=x_1^p+\cdots+x_n^p,~~~ \x=(x_1,\ldots,x_n)\in\mathbb{R}^n,$$
with fixed $0<p<1$, subject to the constraints:
\begin{align}
    g(\x)&:=x_1+\cdots+x_n-n(n-1)=0,\label{g=}\\
    h(\x)&:=x_1^2+\cdots+x_n^2-(n-1)^4-(n-1)\lee0,\label{h<}\\
    d(\x)&:=(n-1)^2-{\textstyle\prod_{i=1}^nx_i}\lee0,\label{prod}\\
    k_i(\x)&:=x_i-(n-1)^2\lee0,~~\hbox{for $i=1,\ldots,n$},\\
    l_i(\x)&:=\xi-x_i\lee0,~~\hbox{for $i=1,\ldots,n$}\label{li},
\end{align}
where $\xi>0$ is fixed so that if for some $i$, $x_i=\xi$, then $\prod_{i=1}^nx_i<(n-1)^2$.

Theorem~\ref{main} now follows if we prove that the minimum   of $f(\x)$ subject to (\ref{g=})--(\ref{li}) is $(n-1)^{2p}+n-1$.

\begin{lem}\label{mfcq} Let $\e$ be a local minimum of $f(\x)$ subject to the constraints (\ref{g=})--(\ref{li}). Then $\e$ satisfies MFCQ.
\end{lem}
\begin{proof}{Let $\e=(e_1,\ldots,e_n)$. With no loss of generality assume that $e_1\g\cdots\g e_n$.
If $e_1=e_n$, then, in view of (\ref{g=}), all $e_i$ are equal to $n-1$. In this case, in none of the inequality constraints (\ref{h<})--(\ref{li}) 
equality occurs for $\e$ and so we are done.
If $e_1>e_n$,  then MFCQ is fulfilled by setting $\w=(-1,0,\ldots,0,1)$.
}\end{proof}

\begin{lem}\label{bennet} {\rm(\cite{ben})} Suppose  $\al,\be,\nu,\omega, a, b, c, d$ are positive numbers and that
\begin{align*}
\al+\be&=\nu+\omega,\\
   \al a+\be b&=\nu c+\omega d,\\
   \max\{a, b\}&\lee\max\{c, d\},\\
   a^\al b^\be&\g c^\nu d^\omega.
\end{align*}
Then the inequality
$$\al a^p+\be b^p\g\nu c^p+\omega d^p$$
holds for  $0\lee p\lee1$.
\end{lem}

\begin{thm}\label{ep} 
Let $\e\in\mathbb{R}^n$ satisfy the constraints (\ref{g=})--(\ref{li}). Then
$f(\e)\g(n-1)^{2p}+n-1$.
\end{thm}
\begin{proof}{It suffices to prove the assertion for local minima. So assume that $\e=(e_1,\ldots,e_n)$ is a local minimum of $f(\x)$ subject to the constraints (\ref{g=})--(\ref{li}).
Suppose that $e_1\g \cdots\g e_n$.
By Lemma~\ref{mfcq}, $\e$ satisfies KKT conditions, namely
\begin{equation}\label{nabla}
   \nabla f(\e)+\mu\nabla g(\e)+\la\nabla h(\e)+\delta\nabla d(\e)+\sum_{i=1}^n\left(\rho_i\nabla k_i(\e)+\gamma_i\nabla l_i(\e)\right)={\bf0},
\end{equation}
\vspace{-.8cm}
\begin{align}
e_1&+\cdots+e_n-n(n-1)=0,\label{sum}\\
\la&\g0,~~~\la h(\e)=0,\label{h}\\
\delta&\g0,~~~\delta d(\e)=0,\nonumber\\
\rho_i&\g0,~~~\rho_i k_i(\e)=0,~~\hbox{for $i=1,\ldots,n$},\label{rho}\\
\gamma_i&\g0,~~~\gamma_i l_i(\e)=0,~~\hbox{for $i=1,\ldots,n$}.\label{gamma}
\end{align}
By the choice of $\xi$ we have $l_i(\e)<0$ for $i=1,\ldots,n$ and hence by (\ref{gamma}), $\gamma_1=\cdots=\gamma_n=0$. If we let $D=\prod_{i=1}^ne_i$, then  (\ref{nabla}) can be written as
$$  pe_i^{p-1}+\mu+2\la e_i-\frac{\delta D}{e_i}+\rho_i=0,~~ \hbox{for $i=1,\ldots,n$}.$$
We consider the following two cases.

\noi{\bf Case 1.} $e_1=(n-1)^2$. Then by (\ref{sum}) and since $\e$ satisfies (\ref{prod}), we have
$$   1=\frac{e_2+\cdots+e_n}{n-1}\g (e_2\cdots e_n)^\frac{1}{n-1}\g1.$$
It turns out that $e_2=\cdots=e_n=1$ and we are done.


\noi{\bf Case 2.} $e_1<(n-1)^2$. So, by (\ref{rho}), $\rho_1=\cdots=\rho_n=0$. It turns out that $e_1,\ldots,e_n$ must satisfy
the following equation:
\begin{equation}\label{pxp}
   px^p=\delta D-\mu x-2\la x^2.
\end{equation}
The curves of $y=px^p$ and $y=\delta D-\mu x-2\la x^2$ intersect in at most two points in $x>0$ and so
 (\ref{pxp}) has at most two positive roots.
 If it has one positive root, then by (\ref{sum}),
$e_1=\cdots=e_n=n-1$.
Hence $f(\e)=n(n-1)^p$ which is greater than $(n-1)^{2p}+n-1$ for $n\g3$. 
Next assume that (\ref{pxp}) has two positive roots, say $a$ and $b$.
These two together with $c=(n-1)^2$ and $d=1$ satisfy the conditions of Lemma~\ref{bennet}.
This implies that $f(\e)\g(n-1)^{2p}+n-1$, completing the proof.
}\end{proof}


\begin{thebibliography}{MM}
\bibitem{agko1} S. Akbari, E. Ghorbani, J.H. Koolen, and M.R. Oboudi, A relation between the Laplacian and signless Laplacian eigenvalues of a graph, {\em J. Algebraic Combin.} {\bf32} (2010),  459--464.
\bibitem{agko2} S. Akbari, E. Ghorbani, J.H. Koolen, and M.R. Oboudi, On sum of powers of the Laplacian and signless Laplacian eigenvalues of graphs, {\em Electron. J. Combin.} {\bf17} (2010), $\#$R 115.
\bibitem{ben} G. Bennett,  $p$-free $\ell^p$ inequalities, {\em Amer. Math. Monthly} {\bf117} (2010), 334--351.
\bibitem{ham} W.H. Haemers, Seidel switching and graph energy, {\em MATCH Commun. Math. Comput. Chem.} {\bf68} (2012), 653--659.
\bibitem{mf} O.L. Mangasarian and S. Fromovitz, The Fritz John necessary optimality conditions in the presence of equality and inequality constraints,
{\em J. Math. Anal. Appl.} {\bf17} (1967), 37--47.
\bibitem{book} J. Nocedal  and S.J. Wright, {\em Numerical Optimization}, Second Edition, Springer, New York, 2006.

\end{thebibliography}
\end{document}